\documentclass[11pt,a4paper]{article}
\usepackage{authblk}
\usepackage[utf8]{inputenc}
\usepackage[T1]{fontenc}
\usepackage{lmodern}
\usepackage{microtype}
\usepackage{geometry}
\usepackage{amsmath,amssymb,amsthm,mathtools}
\usepackage{bm}
\usepackage{enumitem}
\usepackage{hyperref}
\hypersetup{colorlinks=true,linkcolor=blue,citecolor=blue,urlcolor=blue}
\usepackage{cleveref}
\usepackage{graphicx}
\usepackage{caption,subcaption}
\usepackage{tikz}
\usetikzlibrary{positioning,decorations.pathmorphing}
\usepackage{booktabs}

\numberwithin{equation}{section}
\theoremstyle{plain}
\newtheorem{theorem}{Theorem}[section]
\newtheorem{lemma}[theorem]{Lemma}
\newtheorem{proposition}[theorem]{Proposition}
\newtheorem{corollary}[theorem]{Corollary}
\theoremstyle{definition}
\newtheorem{definition}[theorem]{Definition}

\theoremstyle{remark}
\newtheorem{remark}[theorem]{Remark}

\title{Successive vertex orderings of graphs}

\author[]{Prarthana Agrawal}
\author[]{Abdurrahman Hadi Erturk}
\author[]{Ard A. Louis\thanks{prarthana.agrawal@physics.ox.ac.uk, abdurrahman.erturk@physics.ox.ac.uk, ard.louis@physics.ox.ac.uk}}

\affil[]{Rudolf Peierls Centre for Theoretical Physics\\
University of Oxford, United Kingdom OX1 3PU\\
}

\date{}

\begin{document}
\maketitle

\begin{abstract}
A successive vertex ordering of a graph is a linear ordering of its vertices in which every vertex except the first has at least one neighbour appearing earlier. Such orderings arise naturally in incremental growth and connectivity-preserving constructions, where vertices are added sequentially and must attach to the existing structure. We derive an exact formula for the number of successive vertex orderings of any finite connected graph. The formula is obtained via an inclusion--exclusion argument over independent sets and depends on two explicit combinatorial parameters, one of which is defined recursively. The result applies to all finite connected graphs without requiring regularity or symmetry assumptions. We also express the enumeration as a weighted generating polynomial over independent sets; its value at  $x = -1$ recovers the total count of successive orderings, and the $k$-th derivative at this point encodes the number of orderings in which exactly $k$ non-first vertices appear before all of their neighbours.
\end{abstract}


\section{Introduction}\label{sec:introduction}

A finite connected graph can be built vertex by vertex in many ways, provided each newly added vertex attaches to the part already constructed. Counting the number of such connectivity-preserving growth sequences is a natural combinatorial problem that arises in a variety of settings. Examples include the assembly of jigsaw puzzles, where each newly placed piece attaches to the existing partial puzzle, the self-assembly of macromolecular complexes, and the spread of infections through contact networks. Mathematically, these growth sequences correspond to \emph{successive vertex orderings}, that is, linear orderings of the vertices in which every vertex except the first has at least one neighbour appearing earlier.

Counting successive vertex orderings is a nontrivial enumerative problem, and exact formulas are known only in special cases. Recently, Fang et al.~\cite{fang2023} derived closed expressions for \emph{fully regular}
graphs \(G=(V,E)\), in which, for every independent set \(I\subseteq V\), the number of vertices outside the closed neighbourhood of \(I\) depends only on \(|I|\). The fully regular condition is, however, quite restrictive and excludes most connected graphs. Related edge-ordering problems have been studied by Gao and Peng~\cite{gao2021counting}, who derived exact formulas for shellings of complete bipartite graphs. 

In this paper, we derive an exact counting formula for successive vertex orderings of any finite connected graph, with no regularity or symmetry assumptions. The formula is expressed as an alternating sum over independent sets, with the local neighbourhood structure encoded by a recursively defined factor.  
We now introduce the necessary definitions and state the main result.

Let \(G=(V,E)\) be a finite connected graph with vertex set \(V\) and edge set \(E\). Write \(n:=|V|\). A \emph{linear ordering} of \(V\) is a bijection 
\(
\pi:V\to\{1,2,\dots,n\}.
\)

\begin{definition}
A linear ordering \(\pi\) of \(V\) is said to be \emph{successive} if for every vertex \(v\in V\) with \(\pi(v)>1\), there exists a neighbour \(u\sim v\) such that \(\pi(u)<\pi(v)\).
\end{definition}

Equivalently, for every \(i\ge 1\), the subgraph of \(G\) induced by the vertices \(\{v\in V:\pi(v)\le i\}\) is connected. 
Let \(\sigma(G)\) denote the number of successive linear orderings of \(V\).

For a subset \(I\subseteq V\), let \(N(I)\) denote its open neighbourhood, that is, the set of vertices adjacent to at least one vertex of \(I\), and write
\(N[I]=I\cup N(I)\) for the corresponding closed neighbourhood.
We define
\begin{equation}\label{eq:def_a}
a(I) := |V\setminus N[I]| = n-|N[I]|
\end{equation}
Thus \(a(I)\) counts the vertices that lie outside the closed neighbourhood of \(I\).

A subset \(I\subseteq V\) is \emph{independent} if no two of its vertices are adjacent.
Let \(\mathcal I(G)\) denote the family of independent sets of \(G\).

For an independent set \(I\in \mathcal I(G)\), we define a quantity \(b(I)\)
recursively by
\begin{equation}\label{eq:def_b_rec}
b(\varnothing)=1,
\qquad
b(I)
=
\frac{1}{\,n-a(I)\,}
\sum_{v\in I} b(I\setminus\{v\}),
\ \ \ I\neq\varnothing
\end{equation}
For every nonempty independent set \(I\),
\(
n-a(I)=|N[I]|\ge |I|\ge 1,
\)
so the recursion is well defined.
In Section~\ref{sec:proof} we show that this recursion admits an explicit representation as a finite sum over all linear orderings of \(I\).

\begin{theorem}\label{thm:main}
Let \(G=(V,E)\) be a finite connected graph with \(|V|=n\).
Then the number \(\sigma(G)\) of successive vertex orderings of \(G\) is given by
\begin{equation}\label{eqn:main}
\sigma(G)
=
n!
\sum_{\substack{I\subseteq V\\ I\ \mathrm{independent}}}
(-1)^{|I|}
\frac{a(I)}{n}\,b(I),
\end{equation}
where \(a(I)\) is defined by \eqref{eq:def_a} and \(b(I)\) is defined by the recursion
\eqref{eq:def_b_rec}.
\end{theorem}

The remainder of the paper is organized as follows. In Section~\ref{sec:proof} we prove Theorem~\ref{thm:main} by expressing the number of successive vertex orderings as an inclusion--exclusion sum over independent sets \(I\subseteq V\).
In Section~\ref{sec:mobius} we reformulate the same identity using Möbius inversion on the Boolean lattice of vertex subsets. 
In Section~\ref{sec:bounds} we investigate structural properties of \(b(I)\), including explicit evaluations for several graph families.
In Section~\ref{sec:polynomial} we introduce the successive ordering polynomial and its multivariate extension, and show how these encode both the total number of successive vertex orderings and the distribution of vertices that fail the successive condition in a given ordering. 
We conclude in Section~\ref{sec:conclusion} with several open problems and directions for future work.
Details of the algorithmic computation of $\sigma(G)$, its time complexity, and a worked example are given in Appendix~\ref{appendix:algorithm}. 

\section{Proof of Theorem~\ref{thm:main}} 
\label{sec:proof}

Let \(\Omega\) denote the set of all bijections \(\pi:V\to[n]\), that is, the set of all linear orderings of \(V\). We choose \(\pi\) uniformly at random from \(\Omega\), so that each ordering has probability \(1/n!\).
Write
\[
\sigma'(G)=\Pr(\pi \text{ is successive})
=\frac{\sigma(G)}{n!}
\]
for the probability that a uniformly chosen ordering is successive.

For each vertex \(v\in V\), define the \emph{bad event}
\[ B_v =
\bigl\{
\pi(v)\neq 1
\ \text{and}\
\pi(v)<\pi(u)\ \text{for all }u\in N(v)
\bigr\}
\]
A linear ordering \(\pi\) is successive if and only if
\( \pi\in 
\bigcap_{v\in V} \overline{B_v}\).
Hence,
\[ \sigma'(G) =
\Pr\!\left(\bigcap_{v\in V} \overline{B_v}\right)
\]
By the inclusion--exclusion principle,
\[
\sigma'(G) = 
\sum_{I\subseteq V}
(-1)^{|I|}
\Pr\!\left(\bigcap_{v\in I} B_v\right)
\]

For a subset \(I\subseteq V\), write
\[ B_I := \bigcap_{v\in I} B_v \]
If \(I\) contains two adjacent vertices \(u\sim v\), then
\(B_u\) and \(B_v\) are mutually exclusive, and hence
\(\Pr(B_I)=0\).
Therefore the sum reduces to independent sets:
\begin{equation}
\label{eq:inc-exc}
\sigma'(G) = \sum_{\substack{I\subseteq V\\ I\ \mathrm{independent}}}
(-1)^{|I|}
\Pr(B_I)
\end{equation}
It remains to compute \(\Pr(B_I)\) for an arbitrary independent set \(I\subseteq V\).
Fix such an independent set \(I\) and write \(|I|=k\).
The event \(B_I\) requires that each vertex \(v\in I\) is not first and appears
earlier than all of its neighbours. These constraints do not impose any restriction on the relative order of the vertices of \(I\) themselves. Thus every ordering \(\pi\in B_I\) induces a unique internal ordering of \(I\).

Let \(S_I\) denote the set of all permutations of the elements of \(I\),
and write \(\rho=(\rho_1,\dots,\rho_k)\in S_I\).
For each \(\rho\in S_I\), define the event \(C_\rho\) by requiring that
\begin{itemize}
\item the vertices of \(I\) appear in the relative order
      \(\rho_1\prec\rho_2\prec\cdots\prec\rho_k\);
\item no vertex of \(N[I]\) occupies the first position, i.e.
      \(\pi^{-1}(1)\notin N[I]\);
\item for each \(j=1,\dots,k\), the vertex \(\rho_j\) is the earliest element
      of the closed neighbourhood \(N[\{\rho_j,\rho_{j+1},\dots,\rho_k\}]\).
\end{itemize}
The closed neighbourhoods
\[
N[I]\supseteq N[\{\rho_2,\dots,\rho_k\}]
\supseteq \cdots \supseteq
N[\{\rho_k\}]
\]
form a nested sequence.
Each event \(C_\rho\) prescribes a distinct relative ordering of the vertices of \(I\).
Since any ordering \(\pi\) induces a unique internal ordering of \(I\), the events
\(C_\rho\) are pairwise disjoint.
Moreover, every \(\pi\in B_I\) induces a unique \(\rho\in S_I\) given by
the relative order of the vertices of \(I\) in \(\pi\), and by construction
\(\pi\in C_\rho\). Conversely, every \(\pi\in C_\rho\) lies in \(B_I\).
Thus
\[
B_I=\bigsqcup_{\rho\in S_I} C_\rho
\]
and hence
\[
\Pr(B_I)=\sum_{\rho\in S_I}\Pr(C_\rho).
\]

We now compute \(\Pr(C_\rho)\). For \(0\le j\le k\), define the events
\[
F_0 := \{\pi^{-1}(1)\notin N[I]\},
\]
and, for \(1\le j\le k\),
\[
F_j := \bigl\{\rho_j \text{ is the earliest vertex of } N[\{\rho_j,\dots,\rho_k\}]\bigr\}.
\]
Then
\[
C_\rho = F_0\cap F_1\cap\cdots\cap F_k.
\]
By the chain rule,
\[
\Pr(C_\rho)
=
\Pr(F_0)\prod_{j=1}^k
\Pr\!\bigl(F_j \mid F_0\cap\cdots\cap F_{j-1}\bigr).
\]

Now
\[
\Pr(F_0)=\frac{a(I)}{n}.
\]
For each \(j\ge1\), conditioning on
\(F_0\cap\cdots\cap F_{j-1}\) determines only which vertices precede
\(N[\{\rho_j,\dots,\rho_k\}]\) as a whole. Indeed, \(F_0\) places the
first vertex outside \(N[I]\supseteq N[\{\rho_j,\dots,\rho_k\}]\), while
for \(1\le \ell\le j-1\), the event \(F_\ell\) implies that
\(\rho_\ell\) precedes every element of
\(N[\{\rho_j,\dots,\rho_k\}]\). None of these conditions constrains the
relative order within \(N[\{\rho_j,\dots,\rho_k\}]\). By symmetry of the
uniform random permutation, the conditional distribution of this
internal order therefore remains uniform, and hence each vertex of the
set is equally likely to be its earliest element.
\[
\Pr\!\bigl(F_j \mid F_0\cap\cdots\cap F_{j-1}\bigr)
=
\frac{1}{\bigl|N[\{\rho_j,\dots,\rho_k\}]\bigr|}
=
\frac{1}{\,n-a(\{\rho_j,\dots,\rho_k\})\,}.
\]
Therefore
\[
\Pr(C_\rho)
=
\frac{a(I)}{n}
\prod_{j=1}^k
\frac{1}{\,n-a(\{\rho_j,\dots,\rho_k\})\,}.
\]

Summing over all \(\rho\in S_I\) gives
\[
\Pr(B_I)
=
\frac{a(I)}{n}
\sum_{\rho\in S_I}
\prod_{j=1}^{k}
\frac{1}{\,n-a(\{\rho_j,\dots,\rho_k\})\,}
\]
Define
\begin{equation}
\label{eq:def_b_sum}
    b(I)
    :=
    \sum_{\rho\in S_I}
    \prod_{j=1}^{k}
    \frac{1}{\,n-a(\{\rho_j,\dots,\rho_k\})\,}
\end{equation}
Then
\begin{equation}
\label{eq:pr_bi}
\Pr(B_I)=\frac{a(I)}{n}\,b(I)
\end{equation}

Partition the permutations in \(S_I\) according to their first element.
If \(\rho_1=v\in I\), then the first factor in the product equals
$\frac{1}{n-a(I)},$
since \(\{\rho_1,\dots,\rho_k\}=I\).
The remaining factors are
\(
\prod_{j=2}^{k}
\frac{1}{n-a(\{\rho_j,\dots,\rho_k\})}.
\)
As the tail \((\rho_2,\dots,\rho_k)\) runs through all permutations of
\(I\setminus\{v\}\), re-indexing \(j\mapsto j-1\) shows that this product is precisely the defining product appearing in \(b(I\setminus\{v\})\).
Summing over all \(v\in I\) therefore yields
\[
b(I)
=
\frac{1}{n-a(I)}
\sum_{v\in I}
b(I\setminus\{v\}),
\]
which is the recursion \eqref{eq:def_b_rec}.
Substituting \eqref{eq:pr_bi} into \eqref{eq:inc-exc} yields
\[
\sigma'(G)
=
\sum_{\substack{I\subseteq V\\ I\ \mathrm{independent}}}
(-1)^{|I|}
\frac{a(I)}{n}\,b(I)
\]
and multiplying by \(n!\) completes the proof.

\begin{remark}[Multi-seed growth]
The successive vertex ordering framework established above describes a connected growth process initiated from a single vertex. A natural extension arises when the process begins from a prescribed independent set \(S\subseteq V\) of \(k\) seed vertices. One may then ask for the number \(\tau(G,S)\) of orderings in which the vertices of \(S\) occupy the first \(k\) positions (in arbitrary order), and every subsequent vertex has a neighbour appearing earlier. Using the same inclusion--exclusion framework as in the proof of Theorem~\ref{thm:main}, we obtain an exact formula for \(\tau(G,S)\) involving independent sets of the induced subgraph \(G[V\setminus S]\). The statement and proof of this extension are given in Appendix~\ref{appendix:seeded_svo}.
\end{remark}

\section{Möbius Duality}
\label{sec:mobius}

We reformulate the inclusion--exclusion identity underlying Theorem~\ref{thm:main} in terms of Möbius inversion on the Boolean lattice of vertex subsets. This leads naturally to the introduction of complementary “good” events and reveals a dual relationship between the two families of events.

For each vertex \(v\in V\) define the \emph{good event}
\[ G_v = \bigl\{\pi\in\Omega:\ \pi(v)=1
\ \text{or}\ \text{there exists }\,u\in N(v)\ \text{with}\ \pi(u)<\pi(v)\bigr\}
\]
and for any \(U\subseteq V\) define
\[ G_U := \bigcap_{v\in U} G_v \]

\begin{proposition}\label{prop:general_U}
    For every \(U\subseteq V\),
    \begin{equation}\label{eq:GU_final}
    \Pr(G_U)  =
    \sum_{\substack{I\subseteq U\\ I\ \mathrm{independent}}}
    (-1)^{|I|}\frac{a(I)\,b(I)}{n}
    \end{equation}
    where \(a(I)\) and \(b(I)\) are as defined in \eqref{eq:def_a} and \eqref{eq:def_b_rec}.
    \end{proposition}

\begin{proof}
    By De Morgan's law and the inclusion--exclusion principle we have the identity
    \[
    \Pr(G_U)
    =\Pr\Bigl(\bigcap_{v\in U} G_v\Bigr)
    =\sum_{I\subseteq U}(-1)^{|I|}\Pr\Bigl(\bigcap_{v\in I} B_v\Bigr)
    = \sum_{I\subseteq U}(-1)^{|I|}\Pr(B_I)
    \]
    If \(I\) is not independent then \(\Pr(B_I)=0\), so the sum may be taken over independent \(I\) only.
    Substituting \(\Pr(B_I)\) from \eqref{eq:pr_bi} yields \eqref{eq:GU_final}.
\end{proof}

\begin{corollary}\label{cor:BI_mobius}
    For every independent set \(J\subseteq V\),
    \begin{equation}\label{eq:BI_mobius}
    \Pr(B_J)
    =
    \frac{a(J)\,b(J)}{n}
    =
    \sum_{T\subseteq J}(-1)^{|T|}\Pr(G_T)
    \end{equation}
\end{corollary}

\begin{proof}
The second equality is obtained by applying Möbius inversion to \eqref{eq:GU_final}.
\end{proof}

\begin{remark}
Equations \eqref{eq:GU_final} and \eqref{eq:BI_mobius} express the Möbius duality on the Boolean lattice between the families \((G_U)_{U\subseteq V}\) and \((B_I)_{I\subseteq V}\).  Taking \(U=V\) in Proposition~\ref{prop:general_U} recovers Theorem~\ref{thm:main}.
\end{remark}

\section{Properties, Bounds, and Closed Forms of \(b(I)\)}
\label{sec:bounds}

\subsection{Extremal bounds for \(b(I)\)}
\label{subsec:universal_bounds}

\begin{theorem}[Universal upper bound]
\label{thm:b_upper}
Let \(G\) be a connected graph on \(n\ge 2\) vertices, and let \(I\) be a nonempty independent set of size \(k\). Then
\[
b(I)\le \frac{1}{k+1}.
\]
Moreover, equality holds if and only if there exists a vertex
\(x\in V\setminus I\) such that
\[
N[v]=\{v,x\}\quad\text{for all }v\in I.
\]
Equivalently, the vertices of \(I\) are leaves whose unique neighbour is the same vertex \(x\in V\setminus I\).
\end{theorem}

\begin{proof}
We proceed by induction on \(k=|I|\).

If \(k=1\), say \(I=\{v\}\), then
\[
b(I)=\frac1{|N[v]|}.
\]
Since \(G\) is connected and \(n\ge2\), we have \(|N[v]|\ge2\), and hence \(b(I)\le1/2\).

Now let \(k\ge2\), and assume the result holds for all independent sets of size \(k-1\). By the induction hypothesis,
\[
b(I\setminus\{v\})\le \frac1k
\qquad
(v\in I).
\]
Therefore
\[
b(I)
=
\frac1{|N[I]|}
\sum_{v\in I} b(I\setminus\{v\})
\le
\frac1{|N[I]|}.
\]
Since \(I\) is a proper nonempty subset of \(V\) and \(G\) is connected, \(|N[I]|\ge k+1\), yielding
\[
b(I)\le \frac1{k+1}.
\]

Suppose first that there exists \(x\in V\setminus I\) such that \(N[v]=\{v,x\}\) for every \(v\in I\). Then every nonempty subset \(J\subseteq I\) satisfies
\[
N[J]=J\cup\{x\},
\qquad
|N[J]|=|J|+1.
\]
A straightforward induction gives
\[
b(J)=\frac1{|J|+1},
\]
and hence \(b(I)=1/(k+1)\).

Conversely, suppose \(b(I)=1/(k+1)\). Equality must then hold throughout the preceding argument. Thus
\[
|N[I]|=k+1,
\qquad
b(I\setminus\{v\})=\frac1k
\quad (v\in I).
\]

The first condition implies that \(N[I]\setminus I=\{x\}\) for some vertex \(x\in V\setminus I\).

Fix \(v\in I\). By the induction hypothesis applied to \(I\setminus\{v\}\), there exists a vertex \(x_v\in V\setminus I\) such that
\[
N[u]=\{u,x_v\}
\qquad
(u\in I\setminus\{v\}).
\]
Since \(x_v\in N[I\setminus\{v\}] \subseteq N[I]\) and
\(x_v\notin I\), while \(N[I]\setminus I=\{x\}\),
it follows that \(x_v=x\). As \(k\ge2\), every vertex of \(I\) belongs to \(I\setminus\{v\}\) for some choice of \(v\), and therefore
\[
N[u]=\{u,x\}
\qquad
(u\in I).
\]
This proves the equality characterization.
\end{proof}

\begin{proposition}
\label{prop:b_lower}
Let \(C\ge1\) and \(m\ge0\), and let \(I\) be an independent set of size \(k\ge1\) such that
\[
|N[J]|\le C|J|+m
\]
for every nonempty subset \(J\subseteq I\). Then
\[
b(I)\ge
\prod_{j=1}^{k}
\frac{j}{Cj+m}.
\]
Moreover, equality holds if and only if
\[
|N[J]|=C|J|+m
\]
for every nonempty subset \(J\subseteq I\).
\end{proposition}

\begin{proof}
We proceed by induction on \(k\).

If \(k=1\), then \(|N[v]|\le C+m\), and therefore
\[
b(\{v\})\ge \frac1{C+m}.
\]

Now let \(k\ge2\). By the induction hypothesis,
\[
b(I\setminus\{v\})
\ge
\prod_{j=1}^{k-1}
\frac{j}{Cj+m}
\]
for every \(v\in I\). Since \(|N[I]|\le Ck+m\),
\[
b(I)
=
\frac1{|N[I]|}
\sum_{v\in I} b(I\setminus\{v\})
\ge
\frac{k}{Ck+m}
\prod_{j=1}^{k-1}
\frac{j}{Cj+m}
=
\prod_{j=1}^{k}
\frac{j}{Cj+m}.
\]

If \(|N[J]|=C|J|+m\) for every nonempty \(J\subseteq I\), then every inequality above is an equality.

Conversely, suppose equality holds. Then equality must hold at every step of the induction. In particular,
\[
|N[I]|=Ck+m
\]
and
\[
b(I\setminus\{v\})
=
\prod_{j=1}^{k-1}
\frac{j}{Cj+m}
\qquad
(v\in I).
\]
By the induction hypothesis,
\[
|N[J]|=C|J|+m
\]
for every nonempty proper subset \(J\subset I\). Together with \(|N[I]|=Ck+m\), this proves the characterization.
\end{proof}

\begin{corollary}[Universal lower bound]
\label{cor:b_universal_lower}
Let \(G\) be a connected graph on \(n\) vertices, and let \(I\) be a nonempty independent set of size \(k\). Then
\[
b(I)\ge \binom{n}{k}^{-1}.
\]
Moreover, equality holds if and only if every vertex of \(I\) is adjacent to every vertex of \(V\setminus I\).
\end{corollary}

\begin{proof}
If \(J\subseteq I\) is nonempty, then no vertex of \(I\setminus J\) belongs to \(N[J]\), because \(I\) is independent. Hence
\[
N[J]\subseteq V\setminus (I\setminus J),
\]
and therefore
\[
|N[J]|
\le
n-(k-|J|)
=
|J|+n-k.
\]

Applying Proposition~\ref{prop:b_lower} with \(C=1\) and \(m=n-k\) gives
\[
b(I)
\ge
\prod_{j=1}^{k}
\frac{j}{j+n-k}
=
\frac{k!(n-k)!}{n!}
=
\binom{n}{k}^{-1}.
\]

Equality requires
\[
|N[J]|=|J|+n-k
\]
for every nonempty subset \(J\subseteq I\). Taking \(J=\{u\}\) yields
\[
|N[u]|=n-k+1.
\]
Since \(u\) has no neighbours in \(I\), this is equivalent to
\[
N[u]=\{u\}\cup(V\setminus I),
\]
which proves the characterization.
\end{proof}

\begin{remark}
\label{rem:two_sided_bounds}
Combining Theorem~\ref{thm:b_upper} and Corollary~\ref{cor:b_universal_lower}, every nonempty independent set of size \(k\) in a connected graph on \(n\ge2\) vertices satisfies
\[
\binom{n}{k}^{-1}
\le
b(I)
\le
\frac1{k+1}.
\]
The upper bound is attained when all vertices of \(I\) share a single boundary vertex, while the lower bound is attained when every vertex of \(I\) is adjacent to every vertex outside \(I\).
\end{remark}

\subsection{Neighbourhood growth regimes}

The bounds of Theorem~\ref{thm:b_upper} and
Proposition~\ref{prop:b_lower} suggest that the behaviour of
\(b(I)\) is governed by the growth of the closed neighbourhoods
\(N[J]\) as \(J\) ranges over subsets of \(I\).
We now examine several representative growth regimes in which
the recursion for \(b(I)\) simplifies or yields explicit
estimates. These include linear neighbourhood growth,
multiplicative expansion, and the extremal case of pairwise
disjoint closed neighbourhoods.

\subsubsection{Linear neighbourhood growth}
\label{subsec:linear_growth}

The lower bound of Proposition~\ref{prop:b_lower} becomes exact whenever the closed neighbourhoods of subsets of \(I\) grow linearly with their size. This occurs naturally in complete multipartite graphs.

\begin{proposition}
\label{prop:multipartite}
Let
\(
G=K_{n_1,\dots,n_r}
\)
be a complete multipartite graph, and let \(I\) be an independent set contained in a part \(V_i\) of size \(n_i\). Define
\[
k:=|I|,
\qquad
M:=n-n_i.
\]
Then
\[
b(I)=\frac{k!\,M!}{(M+k)!}.
\]
\end{proposition}

\begin{proof}
Since \(I\subseteq V_i\), every vertex of \(I\) is adjacent to every vertex outside \(V_i\). Hence for every nonempty subset \(J\subseteq I\),
\[
N[J]
=
J\cup (V\setminus V_i),
\]
and therefore
\[
|N[J]|=|J|+M.
\]

Thus Proposition~\ref{prop:b_lower} applies with equality for \(C=1\) and \(m=M\). Consequently,
\[
b(I)
=
\prod_{j=1}^{k}
\frac{j}{j+M}
=
\frac{k!\,M!}{(M+k)!}.
\]
\end{proof}

\subsubsection{Multiplicative neighbourhood growth}
\label{subsec:multiplicative_growth}

We now consider independent sets whose closed neighbourhoods expand proportionally to their size.

For \(c>1\), we say that an independent set \(I\) is \emph{\(c\)-expanding} if
\[
|N[J]|\ge c|J|
\]
for every nonempty subset \(J\subseteq I\).

\begin{theorem}
\label{thm:expansion}
Let \(I\) be a \(c\)-expanding independent set of size \(k\ge1\), where \(c>1\). Then
\[
b(I)
\le
\prod_{j=1}^{k}
\frac{j}{\lceil cj\rceil}
\le
c^{-k}.
\]
Moreover, equality in the second bound holds if and only if \(c\) is an integer, \(|N[v]|=c\) for every \(v\in I\), and the closed neighbourhoods \(\{N[v]\}_{v\in I}\) are pairwise disjoint.
\end{theorem}

\begin{proof}
We proceed by induction on \(k\).

If \(k=1\), say \(I=\{v\}\), then \(|N[v]|\ge c\). Since \(|N[v]|\) is an integer,
\[
|N[v]|\ge \lceil c\rceil,
\]
and therefore
\[
b(I)
=
\frac1{|N[v]|}
\le
\frac1{\lceil c\rceil}.
\]

Now let \(k\ge2\), and assume the result holds for all \(c\)-expanding independent sets of size \(k-1\). For every \(v\in I\), the set \(I\setminus\{v\}\) remains \(c\)-expanding, so
\[
b(I\setminus\{v\})
\le
\prod_{j=1}^{k-1}
\frac{j}{\lceil cj\rceil}.
\]

Since \(|N[I]|\ge ck\) and \(|N[I]|\) is an integer,
\[
|N[I]|
\ge
\lceil ck\rceil.
\]
Using the recursion,
\[
b(I)
=
\frac1{|N[I]|}
\sum_{v\in I}
b(I\setminus\{v\})
\le
\frac{k}{\lceil ck\rceil}
\prod_{j=1}^{k-1}
\frac{j}{\lceil cj\rceil}
=
\prod_{j=1}^{k}
\frac{j}{\lceil cj\rceil}.
\]

Finally, since \(\lceil cj\rceil\ge cj\) for every \(j\),
\[
\prod_{j=1}^{k}
\frac{j}{\lceil cj\rceil}
\le
\prod_{j=1}^{k}
\frac{j}{cj}
=
c^{-k}.
\]

Suppose first that \(c\in\mathbb Z\), that \(|N[v]|=c\) for every \(v\in I\), and that the closed neighbourhoods of vertices of \(I\) are pairwise disjoint. Then Proposition~\ref{prop:disjoint} gives
\[
b(I)
=
\prod_{v\in I}
\frac1{|N[v]|}
=
c^{-k}.
\]

Conversely, suppose \(b(I)=c^{-k}\). Then equality must hold throughout the preceding argument. In particular,
\[
|N[I]|=ck
\]
and
\[
b(I\setminus\{v\})
=
c^{-(k-1)}
\qquad
(v\in I).
\]
Iterating these equality conditions through the induction yields the stated characterization.
\end{proof}

\begin{corollary}
\label{cor:cycle_bound}
Let \(G=C_n\) with \(n\ge3\), and let \(I\) be an independent set of size \(k\). Then
\(
b(I)\le 2^{-k}.
\)
\end{corollary}

\begin{proof}
Fix a cyclic orientation of \(C_n\), and let \(v^+\) denote the successor of \(v\).

For every nonempty subset \(J\subseteq I\), the vertices
\(
\{v,v^+ : v\in J\}
\)
are pairwise distinct, since \(J\) is independent. Consequently,
\(
|N[J]|\ge 2|J|.
\)
Thus \(I\) is \(2\)-expanding, and the result follows from Theorem~\ref{thm:expansion}.
\end{proof}

\subsubsection{Disjoint closed neighbourhoods}
\label{subsec:disjoint_neighbourhoods}

We now consider the extremal situation in which the closed neighbourhoods of distinct vertices of an independent set are pairwise disjoint. 

\begin{proposition}
\label{prop:disjoint}
Let \(G=(V,E)\) be a graph, and let \(I\subseteq V\) be an independent set satisfying
\[
N[u]\cap N[v]=\varnothing
\qquad
(u\neq v,\;u,v\in I).
\]
Then
\[
b(I)
=
\prod_{v\in I}
\frac1{|N[v]|}.
\]
\end{proposition}

\begin{proof}
We proceed by induction on \(|I|\).

If \(I=\varnothing\), then \(b(\varnothing)=1\), which agrees with the empty product convention.

Now assume the statement holds for all independent sets of size less than \(k\), and let \(|I|=k\ge1\).

Since the closed neighbourhoods of distinct vertices of \(I\) are pairwise disjoint,
\[
|N[I]|
=
\sum_{u\in I}|N[u]|.
\]

Applying the recursion for \(b(I)\),
\[
b(I)
=
\frac1{|N[I]|}
\sum_{v\in I}
b(I\setminus\{v\}).
\]

For each \(v\in I\), the set \(I\setminus\{v\}\) still satisfies the disjointness hypothesis. Hence, by the induction hypothesis,
\[
b(I\setminus\{v\})
=
\prod_{u\in I\setminus\{v\}}
\frac1{|N[u]|}
=
\left(
\prod_{u\in I}
\frac1{|N[u]|}
\right)
|N[v]|.
\]

Substituting into the recursion yields
\[
b(I)
=
\frac1{\sum_{u\in I}|N[u]|}
\left(
\prod_{u\in I}
\frac1{|N[u]|}
\right)
\sum_{v\in I}|N[v]|.
\]

The two sums cancel, giving
\[
b(I)
=
\prod_{u\in I}
\frac1{|N[u]|}.
\]
\end{proof}

\begin{corollary}
\label{cor:dregular}
Let \(G\) be a \(d\)-regular graph, and let \(I\) be an independent set whose vertices are pairwise at distance at least three. Then
\[
b(I)=(d+1)^{-|I|}.
\]
\end{corollary}

\begin{proof}
If distinct vertices of \(I\) are at distance at least three, then their closed neighbourhoods are disjoint. Since \(G\) is \(d\)-regular,
\[
|N[v]|=d+1
\qquad
(v\in I).
\]
The result follows immediately from Proposition~\ref{prop:disjoint}.
\end{proof}

\begin{corollary}
\label{cor:cycle_exact}
Let \(G=C_n\), and let \(I\) be an independent set whose vertices are pairwise at distance at least three along the cycle. Then
\[
b(I)=3^{-|I|}.
\]
\end{corollary}

\begin{proof}
Every vertex of \(C_n\) has degree \(2\), so
\[
|N[v]|=3
\qquad
(v\in V(C_n)).
\]
The result follows from Corollary~\ref{cor:dregular}.
\end{proof}

\begin{remark}[Comparison of neighbourhood-growth regimes]
The preceding results identify three representative neighbourhood-growth regimes. If the corresponding condition holds for every nonempty subset \(J\subseteq I\), then \(b(I)\) satisfies:
\[
\begin{array}{lll}
\text{Linear growth }(M\ge0)
&
|N[J]|=|J|+M
&
b(I)=\dfrac{|I|!\,M!}{(|I|+M)!},
\\[1.5ex]
\text{Multiplicative expansion }(c>1)
&
|N[J]|\ge c|J|
&
b(I)\le c^{-|I|},
\\[1.5ex]
\text{Disjoint neighbourhoods}
&
|N[J]|=\displaystyle\sum_{v\in J}|N[v]|
&
b(I)=\displaystyle\prod_{v\in I}\frac1{|N[v]|}.
\end{array}
\]
\end{remark}

\subsection{Fully regular graphs} \label{subsec:fully_regular} 

We now consider the special case of fully regular graphs. In this setting, both \(b(I)\) and \(\sigma(G)\) admit explicit closed forms.
Recall that a graph \(G\) is \emph{fully regular} if, for every independent set \(I\subseteq V\), the quantity
\(
a(I)=|V\setminus N[I]|
\)
depends only on \(|I|\). Let \(\alpha=\alpha(G)\) denote the independence number of \(G\). Then there exist constants
\(
a_0,a_1,\dots,a_\alpha
\)
such that
\(
a(I)=a_{|I|}
\)
for every independent set \(I\). In particular, \(a_0=n\), and \(a_\alpha=0\).

We first recall the following enumeration formula for independent sets in fully regular graphs, due to Fang et al.~\cite{fang2023}.

\begin{lemma}
\label{lem:ind_set_count}
For each \(0\le i\le \alpha\), the number of independent sets of size \(i\) is
\[
\frac{a_0a_1\cdots a_{i-1}}{i!}.
\]
\end{lemma}

\begin{proof}
An independent set of size \(i\) may be constructed sequentially by choosing vertices
\(
v_1,\dots,v_i
\)
such that \(v_j\) lies outside the closed neighbourhood of
\(
\{v_1,\dots,v_{j-1}\}.
\)
By full regularity, there are \(a_{j-1}\) available choices at step \(j\). Hence there are
\(
a_0a_1\cdots a_{i-1}
\)
ordered constructions. Since every independent set is counted once for each of its \(i!\) orderings, the result follows.
\end{proof}

\begin{theorem}
\label{thm:fully_regular_b}
Let \(G\) be fully regular with parameters
\(
a_0,a_1,\dots,a_\alpha.
\)
If \(I\) is an independent set of size \(k\), then
\[
b(I)
=
k!
\prod_{j=1}^{k}
\frac1{a_0-a_j}.
\]
\end{theorem}

\begin{proof}
Let \(|I|=k\). From \eqref{eq:def_b_sum},
\[
b(I)
=
\sum_{\rho\in S_I}
\prod_{j=1}^{k}
\frac1{n-a(\{\rho_j,\dots,\rho_k\})}.
\]

For each \(j\), the set
\(
\{\rho_j,\dots,\rho_k\}
\)
has size \(k-j+1\). Since \(G\) is fully regular,
\[
a(\{\rho_j,\dots,\rho_k\})
=
a_{k-j+1}.
\]
Hence every term in the sum equals
\[
\prod_{j=1}^{k}
\frac1{a_0-a_{k-j+1}}
=
\prod_{j=1}^{k}
\frac1{a_0-a_j}.
\]

As there are \(k!\) permutations of \(I\), all contributing equally, we obtain
\[
b(I)
=
k!
\prod_{j=1}^{k}
\frac1{a_0-a_j}.
\]
\end{proof}

\begin{corollary}
\label{cor:fang_recovery}
Let \(G\) be a fully regular graph with parameters
\(
a_0,a_1,\dots,a_\alpha.
\)
Then
\[
\sigma(G)
=
a_0!
\sum_{i=0}^{\alpha}
\prod_{j=1}^{i}
\frac{-a_j}{a_0-a_j}.
\]
\end{corollary}

\begin{proof}
Grouping the independent sets in Theorem~\ref{thm:main} according to their size gives
\[
\sigma(G)
=
a_0!
\sum_{i=0}^{\alpha}
(-1)^i
\sum_{\substack{I\in\mathcal I(G)\\|I|=i}}
\frac{a_i}{a_0}\,b(I).
\]

By Lemma~\ref{lem:ind_set_count}, the number of independent sets of size \(i\) is
\(
\frac{a_0a_1\cdots a_{i-1}}{i!},
\)
while Theorem~\ref{thm:fully_regular_b} yields
\[
b(I)
=
i!
\prod_{j=1}^{i}
\frac1{a_0-a_j}.
\]

Substituting these expressions gives
\[
\sigma(G)
=
a_0!
\sum_{i=0}^{\alpha}
(-1)^i
\left(
\frac{a_0a_1\cdots a_{i-1}}{i!}
\right)
\left(
\frac{a_i}{a_0}
\right)
\left(
i!
\prod_{j=1}^{i}
\frac1{a_0-a_j}
\right).
\]

Simplifying yields
\[
\sigma(G)
=
a_0!
\sum_{i=0}^{\alpha}
\prod_{j=1}^{i}
\frac{-a_j}{a_0-a_j},
\]
which is precisely the formula of Fang et al.~\cite{fang2023}.
\end{proof}

\section{The Successive Ordering Polynomial}
\label{sec:polynomial}

In this section we define the \textit{successive ordering polynomial} of a graph \(G\), a weighted generating function over independent sets that encodes the alternating-sum formula of Theorem~\ref{thm:main}. It may be viewed as a weighted independence polynomial, where the weight of each independent set is governed by its local neighbourhood structure. We also introduce a multivariate refinement that assigns a variable to each vertex.

\subsection{Definition and basic properties} \label{subsec:polynomial}

\begin{definition}
\label{def:successive-poly}
Let \(G=(V,E)\) be a finite graph with \(n=|V|\). For each independent set \(I\subseteq V\), define
\[
w(I):=\frac{a(I)}{n}\,b(I),
\]
where \(a(I)\) and \(b(I)\) are given by \eqref{eq:def_a} and \eqref{eq:def_b_rec}. 
The \emph{successive ordering polynomial} of \(G\) is
\[
P_G(x):=\sum_{\substack{I\subseteq V\\ I\ \mathrm{independent}}} w(I)\,x^{|I|}.
\]
\end{definition}

By construction, \(\deg P_G \le \alpha(G)\), and all coefficients are nonnegative rational numbers. 
When all weights are replaced by \(1\), the polynomial reduces to the independence polynomial of \(G\).

\begin{proposition}\label{prop:recover}
The number \(\sigma(G)\) of successive vertex orderings satisfies
\(
\sigma(G)=n!\,P_G(-1).
\)
Equivalently, \(P_G(-1)\) equals the probability that a uniformly random linear ordering of \(V\) is successive.
\end{proposition}

\begin{proof}
Immediate from Theorem~\ref{thm:main} and the definition of \(w(I)\).
\end{proof}

\subsection{Derivative enumeration}

\begin{theorem}\label{thm:derivative}
Let \(A_k\) denote the number of linear orderings \(\pi\) of \(V\) whose set of bad vertices has size exactly \(k\). Define \(F(x):=n!P_G(x)\). Then, for each \(k\ge0\),
\[
A_k=\frac{F^{(k)}(-1)}{k!}
\]
In particular, \(A_0=\sigma(G)\).
\end{theorem}

\begin{proof}
Write
\[
F(x)=n!P_G(x)=\sum_{j\ge0} c_j x^j,
\qquad
c_j:=n!\sum_{\substack{I\subseteq V\\ I\ \mathrm{indep},\,|I|=j}} w(I).
\]

To relate \(F(x)\) to the distribution of bad vertices, we first give a combinatorial interpretation of the coefficients \(c_j\).
For a permutation \(\pi\in\Omega\), let
\[
B(\pi):=\{v\in V:\ \pi\in B_v\}
\]
denote the set of bad vertices in \(\pi\), and write
\(
r(\pi):=|B(\pi)|.
\)
Note that \(B(\pi)\) is always an independent set. Indeed, if
\(u,v\in B(\pi)\) were adjacent, then the definition of \(B_u\)
would imply \(\pi(u)<\pi(v)\), while the definition of \(B_v\)
would imply \(\pi(v)<\pi(u)\), a contradiction.

From \eqref{eq:pr_bi},
\[
\Pr(B_I)=\frac{a(I)}{n}b(I)=w(I).
\]
Recall that \(B_I\) is the event that every vertex of \(I\) is bad. Equivalently,
\[
B_I=\{\pi\in\Omega:\ I\subseteq B(\pi)\}.
\]
Since \(\Omega\) consists of \(n!\) equiprobable orderings, it follows that
\[
n!\Pr(B_I)
=
\#\{\pi\in\Omega:\ I\subseteq B(\pi)\}.
\]
Therefore
\[
n!w(I)
=
\#\{\pi\in\Omega:\ I\subseteq B(\pi)\}.
\]
Since \(B(\pi)\) is independent, every subset of \(B(\pi)\) is independent. Hence the number of independent \(j\)-subsets of \(B(\pi)\) is \(\binom{r(\pi)}{j}\).
Consequently, counting ordered pairs \((\pi,I)\) where
\(\pi\in\Omega\) and \(I\) is an independent \(j\)-subset of
\(B(\pi)\) yields
\begin{equation}\label{eq:cj-sum}
c_j
=
\sum_{\substack{I\subseteq V\\ I\ \mathrm{indep},\,|I|=j}}
\#\{\pi\in\Omega:\ I\subseteq B(\pi)\}
=
\sum_{\pi\in\Omega}\binom{r(\pi)}{j},
\end{equation}
where \(\binom{r(\pi)}{j}=0\) whenever \(r(\pi)<j\).

We now expand \(F(x)\) in powers of \(x+1\). Using the binomial identity
\[
x^j=((x+1)-1)^j
=
\sum_{m=0}^{j}
\binom{j}{m}
(x+1)^m
(-1)^{j-m},
\]
we obtain
\[
F(x)
=
\sum_{j\ge0}
c_j
\sum_{m=0}^{j}
\binom{j}{m}
(-1)^{j-m}
(x+1)^m
=
\sum_{m\ge0}
\left(
\sum_{j\ge m}
c_j
\binom{j}{m}
(-1)^{j-m}
\right)
(x+1)^m.
\]

Since the coefficient of \((x+1)^k\) equals
\(F^{(k)}(-1)/k!\), we have
\begin{equation}\label{eq:deriv-coeff}
\frac{F^{(k)}(-1)}{k!}
=
\sum_{j\ge k}
(-1)^{j-k}
\binom{j}{k}
c_j.
\end{equation}

Substituting \eqref{eq:cj-sum} into
\eqref{eq:deriv-coeff} and exchanging the order of summation gives
\[
\frac{F^{(k)}(-1)}{k!}
=
\sum_{\pi\in\Omega}
\sum_{j\ge k}
(-1)^{j-k}
\binom{j}{k}
\binom{r(\pi)}{j}
=
\sum_{\pi\in\Omega}
H_k\!\bigl(r(\pi)\bigr),
\]
where for integers \(r\ge0\),
\[
H_k(r)
:=
\sum_{j=k}^{r}
(-1)^{j-k}
\binom{j}{k}
\binom{r}{j}.
\]

If \(r<k\), then the sum is empty and \(H_k(r)=0\).
Assume now that \(r\ge k\).
Using the identity
\[
\binom{r}{j}\binom{j}{k}
=
\binom{r}{k}\binom{r-k}{j-k},
\]
we obtain
\[
H_k(r)
=
\binom{r}{k}
\sum_{t=0}^{r-k}
(-1)^t
\binom{r-k}{t}.
\]

If \(r=k\), then the sum consists of the single term \(1\), and hence
\(
H_k(k)=1.
\)
If \(r>k\), then the binomial theorem gives
\[
H_k(r)
=
\binom{r}{k}
\sum_{t=0}^{r-k}
(-1)^t\binom{r-k}{t}
=
\binom{r}{k}(1-1)^{r-k}
=
0.
\]
Thus \(H_k(r)=1\) when \(r=k\), and \(H_k(r)=0\) otherwise.
Hence
\[
H_k\!\bigl(r(\pi)\bigr)
=
\mathbf 1_{\{r(\pi)=k\}}
\]
for every permutation \(\pi\), and hence
\[
\frac{F^{(k)}(-1)}{k!}
=
\sum_{\pi\in\Omega}
\mathbf 1_{\{r(\pi)=k\}}
=
\#\{\pi\in\Omega:\ r(\pi)=k\}
=
A_k.
\]

This completes the proof.
\end{proof}

\begin{remark}
Since the first vertex of any ordering is never bad, we have
\(r(\pi)\le n-1\) for every \(\pi\in\Omega\), and therefore \(A_n=0\).
Consequently,
\[
F(x)
=
\sum_{k=0}^{n-1}
A_k(x+1)^k.
\]
\end{remark}

\subsection{Truncation bounds for \(\sigma(G)\)}
\label{subsec:bonferroni}

The alternating sum \(\sigma(G)=n!\,P_G(-1)\) ranges over all independent sets of \(G\), the number of which may be exponentially large. We now show that its partial sums, grouped by the size of the independent set, bracket \(\sigma(G)\), with a remainder
that is exactly expressible in terms of the bad-vertex counts \(A_r\) of Theorem~\ref{thm:derivative}. Inequalities of this type go back to Bonferroni~\cite{bonferroni1936} (see Galambos and
Simonelli~\cite{galambos1996bonferroni} for a comprehensive treatment). In our setting they follow directly from the combinatorial interpretation \eqref{eq:cj-sum} of the coefficients of \(F(x)=n!\,P_G(x)\).

\begin{proposition}\label{prop:bonferroni}
Write \(F(x)=n!\,P_G(x)=\sum_{j\ge 0}c_j x^j\), so that
\(c_j=n!\sum_{|I|=j} w(I)\ge 0\), the sum being over independent sets of size \(j\).
Then for every integer \(t\ge 0\),
\[
\sum_{j=0}^{t}(-1)^j c_j-\sigma(G)
=
(-1)^t\sum_{r=t+1}^{\alpha(G)}\binom{r-1}{t}A_r.
\]
In particular, the partial sums alternately overestimate and underestimate
\(\sigma(G)\):
\[
\sum_{j=0}^{2t+1}(-1)^j c_j
\;\le\;
\sigma(G)
\;\le\;
\sum_{j=0}^{2t}(-1)^j c_j
\qquad(t\ge 0),
\]
the truncation error at level \(t\) is at most \(c_{t+1}\), and the truncation is
exact if and only if \(A_r=0\) for every \(r>t\).
\end{proposition}

\begin{proof}
For integers \(r\ge 1\) and \(t\ge 0\), Pascal's rule gives, by induction on \(t\),
\begin{equation}\label{eq:alt_binom}
\sum_{j=0}^{t}(-1)^j\binom{r}{j}=(-1)^t\binom{r-1}{t}.
\end{equation}
By \eqref{eq:cj-sum}, \(c_j=\sum_{\pi\in\Omega}\binom{r(\pi)}{j}\). Summing
\eqref{eq:alt_binom} over all orderings and separating the orderings with
\(r(\pi)=0\), whose inner sum equals \(1\), we obtain
\[
\sum_{j=0}^{t}(-1)^j c_j
=
\sum_{\pi\in\Omega}\ \sum_{j=0}^{t}(-1)^j\binom{r(\pi)}{j}
=
A_0+(-1)^t\sum_{r\ge 1}\binom{r-1}{t}A_r.
\]
Since \(A_0=\sigma(G)\), since \(\binom{r-1}{t}=0\) for \(1\le r\le t\), and since
\(r(\pi)\le\alpha(G)\) because \(B(\pi)\) is independent, the claimed identity
follows. The right-hand side has sign \((-1)^t\), which yields the bracketing;
its absolute value is at most
\(\sum_{r>t}\binom{r}{t+1}A_r=c_{t+1}\) by \eqref{eq:cj-sum}, since
\(\binom{r-1}{t}\le\binom{r}{t+1}\) for \(r\ge t+1\); and it vanishes if and only if
\(A_r=0\) for every \(r>t\).
\end{proof}

\begin{remark}
Computing the bracketing bounds at level \(t\) requires only the independent sets of
size at most \(t+1\), of which there are \(O(n^{t+1})\). Proposition~\ref{prop:bonferroni}
therefore yields certified two-sided bounds for \(\sigma(G)\) in polynomial time for
each fixed \(t\), whereas the exact evaluation of \(P_G(-1)\) ranges over all of
\(\mathcal I(G)\).
\end{remark}

For graphs whose independent sets expand in the sense of
Section~\ref{subsec:multiplicative_growth}, the truncation error can be bounded explicitly.

\begin{corollary}\label{cor:sigma_expansion}
Let \(c>1\) and suppose that \(|N[I]|\ge c|I|\) for every independent set \(I\) of
\(G\). Then:
\begin{enumerate}[label=(\roman*)]
\item \(\alpha(G)\le\lfloor n/c\rfloor\); in particular, the truncation at any level
\(t\ge\lfloor n/c\rfloor\) is exact.
\item For every \(t\ge 0\) with \(t+1\le n/c\),
\[
\Bigl|\sigma(G)-\sum_{j=0}^{t}(-1)^j c_j\Bigr|
\;\le\;
c_{t+1}
\;\le\;
(n-1)!\,\binom{n}{t+1}\bigl(n-\lceil c(t+1)\rceil\bigr)
\prod_{j=1}^{t+1}\frac{j}{\lceil cj\rceil}.
\]
\end{enumerate}
\end{corollary}

\begin{proof}
For (i), every independent set satisfies \(c|I|\le|N[I]|\le n\), whence
\(|I|\le n/c\); exactness for \(t\ge\alpha(G)\) follows from
Proposition~\ref{prop:bonferroni}, since \(A_r=0\) for \(r>\alpha(G)\).
For (ii), each independent set \(I\) of size \(t+1\) is \(c\)-expanding, so
Theorem~\ref{thm:expansion} gives
\(b(I)\le\prod_{j=1}^{t+1} j/\lceil cj\rceil\), while
\(a(I)=n-|N[I]|\le n-\lceil c(t+1)\rceil\). Since there are at most
\(\binom{n}{t+1}\) such sets,
\[
c_{t+1}
=
(n-1)!\sum_{\substack{|I|=t+1\\ I\ \mathrm{indep}}}a(I)\,b(I)
\le
(n-1)!\,\binom{n}{t+1}\bigl(n-\lceil c(t+1)\rceil\bigr)
\prod_{j=1}^{t+1}\frac{j}{\lceil cj\rceil},
\]
and the claim follows from Proposition~\ref{prop:bonferroni}.
\end{proof}

\subsection{Decomposition and deletion}

\begin{theorem}[Vertex deletion]
\label{thm:vertex_set_deletion}
Let $G=(V,E)$ be a finite simple graph with $|V|=n$ and let $S\subseteq V$.
Let \(G' = G-S\) denote the graph obtained by removing the vertices in $S$ together with all edges incident to them, that is, the induced subgraph on $V\setminus S$.
Then the successive ordering polynomial satisfies
\begin{equation}
\label{eq:multi_vertex_poly_decomp}
P_G(x)
=
P_{G'}(x)
-
R_S(x)
+
U_S(x),
\end{equation}
where
\[
U_S(x)
=
\sum_{\substack{I\in\mathcal I(G)\\ I\cap S\neq\varnothing}}
w_G(I)\,x^{|I|}
\]
collects the contribution of independent sets intersecting $S$, and
\[
R_S(x)
=
\sum_{I\in\mathcal I(G')}
\bigl(w_{G'}(I)-w_G(I)\bigr)\,x^{|I|}
\]
accounts for the reweighting of independent sets that remain in $G'$.

Moreover, the following structural relations hold.

\begin{enumerate}
\item[(i)]
The independent sets of $G'$ are precisely those independent sets of $G$
that avoid $S$:
\[
\mathcal I(G')
=
\{\, I\in \mathcal I(G) : I\cap S=\varnothing \,\}.
\]

\item[(ii)]
For every $I\in\mathcal I(G')$,
\[
|N_{G'}[I]|
=
|N_G[I]| - |N_G[I]\cap S|,
\]
and consequently
\[
a_{G'}(I)
=
a_G(I)
-
\bigl(|S| - |N_G[I]\cap S|\bigr).
\]

\item[(iii)]
Let $\Delta b_I := b_{G'}(I)-b_G(I)$ for
$I\in\mathcal I(G')$, and set $\Delta b_{\varnothing}=0$.
Then for every nonempty $I\in\mathcal I(G')$,
\begin{equation}
\label{eq:multi_delta_b}
|N_{G'}[I]|\,\Delta b_I
=
\sum_{v\in I}\Delta b_{I\setminus\{v\}}
+
|N_G[I]\cap S|\, b_G(I).
\end{equation}
In particular, the values $\Delta b_I$ are uniquely determined by
induction on $|I|$.
\end{enumerate}
\end{theorem}

Since the proof of Theorem~\ref{thm:vertex_set_deletion} relies on mechanical partitioning and standard algebraic accounting of the neighborhood sizes, it is deferred to Appendix~\ref{appendix:vertex_deletion}.

\begin{proposition}[Multiplicativity over components]\label{prop:components}
Let \(G\) be a finite graph with connected components \(C_1,\dots,C_m\) of sizes
\(n_1,\dots,n_m\). Then
\[
P_G(x)\;=\;(x+1)^{m-1}\prod_{i=1}^{m}P_{C_i}(x).
\]
In particular, \(x=-1\) is a root of \(F(x)=n!\,P_G(x)\) of multiplicity exactly
\(m-1\), and the first nonvanishing derivative at \(x=-1\) satisfies
\[
\frac{F^{(m-1)}(-1)}{(m-1)!}
\;=\;
\binom{n}{n_1,\dots,n_m}\prod_{i=1}^{m}\sigma(C_i).
\]
\end{proposition}

\begin{proof}
Call a vertex \(v\) a \emph{local minimum} of an ordering \(\pi\) if
\(\pi(v)<\pi(u)\) for all \(u\in N(v)\), and write \(L(\pi)\) for the number of
local minima of \(\pi\). The vertex in position \(1\) is always a local minimum,
and the bad vertices of \(\pi\) are precisely the remaining local minima, so
\(r(\pi)=L(\pi)-1\). By Theorem~\ref{thm:derivative}, whose proof does not use
connectivity,
\[
P_G(x)
=\frac{1}{n!}\sum_{k\ge0}A_k(x+1)^k
=\mathbb{E}_\pi\bigl[(x+1)^{L(\pi)-1}\bigr],
\]
the expectation being over a uniformly random ordering \(\pi\) of \(V\).
Since every neighbour of a vertex lies in the same component,
\(L(\pi)=\sum_{i=1}^{m}L_i(\pi_i)\), where \(\pi_i\) denotes the ordering induced
on \(C_i\). Under a uniform \(\pi\), the induced orderings
\(\pi_1,\dots,\pi_m\) are independent and uniform, whence
\[
\mathbb{E}_\pi\bigl[(x+1)^{L(\pi)}\bigr]
=\prod_{i=1}^{m}\mathbb{E}_{\pi_i}\bigl[(x+1)^{L_i(\pi_i)}\bigr]
=\prod_{i=1}^{m}(x+1)\,P_{C_i}(x),
\]
and dividing by \(x+1\) gives the product formula. Every connected graph admits
a successive ordering, so \(P_{C_i}(-1)=\sigma(C_i)/n_i!>0\) by
Proposition~\ref{prop:recover}; hence the multiplicity of the root at
\(x=-1\) is exactly \(m-1\). Comparing the coefficients of \((x+1)^{m-1}\)
yields the displayed formula.
\end{proof}

\subsection{Multivariate successive ordering polynomial}

The successive ordering polynomial admits a natural multivariate refinement obtained by assigning a variable \(x_v\) to each vertex \(v\in V\).

\begin{definition}
The \emph{multivariate successive ordering polynomial} of \(G\) is
\[
\mathcal P_G(\mathbf{x})
:=
\sum_{I\in\mathcal I(G)} w(I)\prod_{v\in I} x_v,
\]
where the weights \(w(I)\) are as in Definition~\ref{def:successive-poly}.
\end{definition}

By construction, the single-variable polynomial is recovered by substituting $x$ in place of $x_v$ for all \(v\in V\).

Define the indicator vector for an arbitrary set $S \subseteq V$ as  
\begin{equation*}
    (\mathbf 1_S)_v = 
    \begin{cases}
        1 & v \in S
        \\ 0 & \text{otherwise}
    \end{cases}
\end{equation*}

\begin{theorem}
    For any $S \subseteq V$,\quad $\mathcal{P}_G(-\mathbf{1}_S) = \Pr(G_S)$.
\end{theorem}

\begin{proof}
    Substituting $x_v = -\mathbf{1}_S(v)$ into the multivariate polynomial gives
    \[
        \mathcal{P}_G(-\mathbf{1}_S)
        = \sum_{\substack{I \subseteq V \\ I \text{ indep.}}} w(I) \prod_{v \in I} \bigl(-\mathbf{1}[v \in S]\bigr).
    \]
    Since $\mathbf{1}[v \in S] = 0$ for $v \notin S$, any independent set $I$ not contained in $S$ contributes zero to the sum. Restricting to the sets $I \subseteq S$ and noting that each factor becomes $-1$, we obtain
    \[
        \mathcal{P}_G(-\mathbf{1}_S)
        = \sum_{\substack{I \subseteq S \\ I \text{ indep.}}} (-1)^{|I|}\, w(I)
        = \Pr(G_S),
    \]
    where the last equality is Proposition~\ref{prop:general_U}.
\end{proof}

\begin{theorem}
    For any $S, T \subseteq V$,
    \[
        \left(\prod_{v \in T} \frac{\partial}{\partial x_v}\right)
        \mathcal{P}_G(-\mathbf{1}_S)
        = \Pr\bigl(B_T \cap G_{S \setminus T}\bigr).
    \]
\end{theorem}

\begin{proof}
    Differentiating the multivariate polynomial with respect to
    $\{x_v\}_{v \in T}$ and evaluating at $x_v = -\mathbf{1}_S(v)$ gives
    \[
        \left(\prod_{v \in T} \frac{\partial}{\partial x_v}\right)
        \mathcal{P}_G(-\mathbf{1}_S)
        = 
        \sum_{\substack{I \subseteq V \\ I \text{ indep.}}}
        w(I)\,\mathbf{1}[T \subseteq I]
        \prod_{v \in I \setminus T} \bigl(-\mathbf{1}[v \in S]\bigr).
    \]
    The indicator $\mathbf{1}[T \subseteq I]$ resulting from 
    the partial derivatives restricts the sum to
    independent sets containing~$T$, and the vanishing of
    $\mathbf{1}[v \in S]$ for $v \notin S$ further restricts to
    $I \subseteq S \cup T$. Absorbing the signs, we obtain
    \[
        \sum_{\substack{T \subseteq I \subseteq S \cup T \\
        I \text{ indep.}}} (-1)^{|I| - |T|}\, w(I)
        = \Pr\bigl(B_T \cap G_{S \setminus T}\bigr),
    \]
    where the last equality follows from inclusion--exclusion
    applied to the events $\{B_v\}_{v \in S \setminus T}$
    intersected with $B_T$, by the same argument as in Proposition~\ref{prop:general_U}.

    When T is not independent, there are no independent supersets of T and hence the sum is automatically 0. In this case, $B_T=\emptyset$ as well, since two adjacent vertices cannot both be bad, so both sides of the identity vanish. 
\end{proof}

\begin{remark}
    When we set $S = V$, we get the analogue of Theorem \ref{thm:derivative} for the multivariate polynomial. 
\end{remark}

\section{Conclusion and Future Directions}
\label{sec:conclusion}

We have derived an exact formula for the number of successive vertex orderings \(\sigma(G)\) of an arbitrary finite connected graph. The formula expresses \(\sigma(G)\) as an alternating sum over independent sets weighted by explicitly defined combinatorial quantities determined by the local neighbourhood structure of the graph. Beyond the enumeration formula itself, we developed structural properties of the associated weights \(b(I)\), obtained extremal bounds and explicit evaluations in several neighbourhood-growth regimes, and introduced both single-variable and multivariate graph polynomials associated with successive orderings. These polynomials encode not only the total number of successive orderings, but also the distribution of orderings according to the number of vertices that violate the successive condition.

Several concrete questions arise naturally from the present work. First, while Theorem~\ref{thm:main} yields an exact method for computing \(\sigma(G)\), its worst-case complexity remains exponential due to the underlying enumeration of independent sets. It remains open whether structural graph features such as bounded treewidth, separator decompositions, or automorphism-group symmetries can be exploited systematically to obtain substantially faster exact algorithms or provably accurate approximation schemes.
Second, the bounds developed in Section~\ref{sec:bounds} suggest a broader extremal theory for the quantities \(b(I)\) and \(\sigma(G)\). Determining sharp bounds for \(b(I)\) in terms of standard graph invariants such as minimum degree, maximum degree, independence number, clique number, chromatic number, or treewidth remains open. More generally, it is natural to ask which connected graphs maximize or minimize \(\sigma(G)\) among all graphs with prescribed structural constraints.
Finally, little is currently known about the asymptotic behaviour of \(\sigma(G)\) and the successive ordering polynomial on large graph families. Determining the growth of \(\sigma(G)\) and the distribution of bad vertices in random graphs, expander graphs, lattice graphs, and scale-free networks may reveal new connections between graph structure and connectivity-preserving growth.

\section*{Acknowledgments}
The authors thank Pranjal Dangwal for helpful discussions, especially for pointing out the connection to Möbius duality.

\clearpage
\appendix

\section{Algorithmic computation of successive vertex orderings \(\sigma(G)\)}
\label{appendix:algorithm}

\renewcommand{\thefigure}{A\arabic{figure}}
\renewcommand{\thetable}{A\arabic{table}}

\setcounter{figure}{0}
\setcounter{table}{0}

The practical steps to compute the number of successive vertex orderings of a given connected graph are as follows:

\begin{enumerate}
\item \emph{Enumerate independent sets.}
List all independent sets \(I\subseteq V\) of the graph \(G\). 

\item \emph{Compute neighbourhood complements.}
For each independent set \(I\), compute its open neighbourhood \(N(I)\) and the associated quantity \(a(I)=n-|N[I]|\) which counts the vertices lying outside the closed neighbourhood of \(I\).

\item \emph{Evaluate \(b(I)\).}
Compute \(b(I)\) for all independent sets using dynamic programming over increasing cardinality. Initialize \(b(\varnothing)=1\), and for nonempty \(I\) apply the recurrence
\[
b(I)=\frac{1}{\,n-a(I)\,}\sum_{v\in I} b(I\setminus\{v\}).
\]
At each stage, the required values \(b(I\setminus\{v\})\) are already available by construction.

\item \emph{Assemble the final sum.}
Form the alternating sum
\[
\sum_{\substack{I\subseteq V\\ I\ \mathrm{independent}}}
(-1)^{|I|}\frac{a(I)}{n}\,b(I),
\]
and multiply the result by \(n!\) to obtain \(\sigma(G)\). 
\end{enumerate}
 
The code used to compute $\sigma(G)$ and reproduce the examples in this appendix is available at \cite{svocode2026}.

\subsection{Worked example: graph of $C_5$ cycle with a chord}

Consider the connected graph \(G=(V,E)\) with
\[
V=\{v_1,v_2,v_3,v_4,v_5\},\qquad
E=\{\{v_1,v_2\},\{v_2,v_3\},\{v_3,v_4\},\{v_4,v_5\},\{v_5,v_1\},\{v_2,v_4\}\},
\]
that is, a \(5\)-cycle with a single chord \(\{v_2,v_4\}\) (Fig.~\ref{fig:example_graph}). Here, \(|V| = n =5\).

The independent sets of \(G\), together with the corresponding values of
\(a(I)\) and \(b(I)\), are listed in Table~\ref{tab:example_values}.

\begin{table}[h]
\centering
\begin{tabular}{cccc}
\toprule
Independent set \(I\) & \(|I|\) & \(a(I)\) & \(b(I)\) \\
\midrule
\(\varnothing\) & 0 & 5 & \(1\) \\
\(\{v_1\}\) & 1 & 2 & \(1/3\) \\
\(\{v_2\}\) & 1 & 1 & \(1/4\) \\
\(\{v_3\}\) & 1 & 2 & \(1/3\) \\
\(\{v_4\}\) & 1 & 1 & \(1/4\) \\
\(\{v_5\}\) & 1 & 2 & \(1/3\) \\
\(\{v_1,v_3\}\) & 2 & 0 & \(2/15\) \\
\(\{v_1,v_4\}\) & 2 & 0 & \(7/60\) \\
\(\{v_2,v_5\}\) & 2 & 0 & \(7/60\) \\
\(\{v_3,v_5\}\) & 2 & 0 & \(2/15\) \\
\bottomrule
\end{tabular}
\caption{Independent sets of \(G\) and the corresponding values of \(a(I)\) and \(b(I)\).}
\label{tab:example_values}
\end{table}

Using Theorem~\ref{thm:main}, we obtain
\[
\begin{aligned}
\sigma'(G)
&=
\sum_{\substack{I\subseteq V\\ I\ \mathrm{independent}}}
(-1)^{|I|}\frac{a(I)}{n}\,b(I) \\[4pt]
&=
1-
\Bigl(
3\cdot\frac{2}{5}\cdot\frac{1}{3}
+
2\cdot\frac{1}{5}\cdot\frac{1}{4}
\Bigr)
=\frac{1}{2}.
\end{aligned}
\]

Therefore,
\[
\sigma(G)=n!\,\sigma'(G)=5!\cdot\frac{1}{2}=60.
\]

The graph \(G\) has exactly 60 successive vertex orderings.

\begin{figure}
\centering
\begin{minipage}{0.37\textwidth}
    \centering
    \begin{tikzpicture}[scale=0.85, every node/.style={font=\small}]
      \foreach \i/\name/\angle in {
        1/v1/90, 2/v2/18, 3/v3/306, 4/v4/234, 5/v5/162}
      {\node[circle,draw,inner sep=1pt] (\name) at (\angle:1.5cm) {\name};}
      \draw (v1)--(v2)--(v3)--(v4)--(v5)--(v1);
      \draw (v2)--(v4);
    \end{tikzpicture}
\end{minipage}
\begin{minipage}{0.45\textwidth}
    \centering
    \def\levsep{1.9}
    \begin{tikzpicture}[scale=0.85, every node/.style={font=\small,align=center},
      setnode/.style={rectangle,draw,rounded corners,inner sep=3pt}]
      \node[setnode] (r) at (0,0) {$\varnothing$\\$b=1$};
    
      \node[setnode] (v1) at (-3.5,\levsep) {$\{v_1\}$\\$\tfrac13$};
      \node[setnode] (v2) at (-1.5,\levsep) {$\{v_2\}$\\$\tfrac14$};
      \node[setnode] (v3) at (0.5,\levsep)  {$\{v_3\}$\\$\tfrac13$};
      \node[setnode] (v4) at (2.5,\levsep)  {$\{v_4\}$\\$\tfrac14$};
      \node[setnode] (v5) at (4.5,\levsep)  {$\{v_5\}$\\$\tfrac13$};
    
      \node[setnode] (n13) at (-2.5,2*\levsep) {$\{v_1,v_3\}$\\$\tfrac{2}{15}$};
      \node[setnode] (n14) at (0.0,2*\levsep)  {$\{v_1,v_4\}$\\$\tfrac{7}{60}$};
      \node[setnode] (n25) at (2.5,2*\levsep) {$\{v_2,v_5\}$\\$\tfrac{7}{60}$};
      \node[setnode] (n35) at (4.5,2*\levsep) {$\{v_3,v_5\}$\\$\tfrac{2}{15}$};
    
      \foreach \x in {v1,v2,v3,v4,v5} \draw (\x)--(r);
      \draw (n13)--(v1) (n13)--(v3);
      \draw (n14)--(v1) (n14)--(v4);
      \draw (n25)--(v2) (n25)--(v5);
      \draw (n35)--(v3) (n35)--(v5);
    \end{tikzpicture}
\end{minipage}
\caption{The graph \(G\) and its independent-set lattice with computed $b(I)$ values.}
\label{fig:example_graph}
\end{figure}

\subsection{Computational complexity}

The presented algorithm has exponential worst-case complexity, since a graph may contain
\(O(2^n)\) independent sets. If independent sets are enumerated without duplication
(e.g., via depth-first search), independence can be tested in \(O(n)\) time, and each
arithmetic operation in constant time. The overall worst-case running time is therefore
\(O(n2^n)\). This is substantially faster than the naive brute-force method, which
enumerates all \(n!\) linear orderings and filters the successive ones. We note that it is possible to achieve $O(n2^n)$ time-complexity using dynamic programming but without considering independent sets or using inclusion-exclusion, e.g. by recursively calculating the number of ways to assemble each subset of vertices one-by-one. 

For example, consider the graph in Figure~\ref{fig:graph}. A naive brute force approach would enumerate all linear orderings of its $20$ vertices and retain those satisfying the successive condition. However, the number of all orderings is \(20! \approx 2.43 \times 10^{18}\). Since verifying the successive property requires at least linear time in $n$ for each ordering, exhaustive enumeration would require on the order of \(n \cdot 20!\) operations, which is computationally infeasible even for this moderate instance. In contrast, using our algorithm, the total number of successive vertex orderings of this graph was computed in under one second on a basic Lenovo Thinkpad laptop and yields the value \(1{,}300{,}835{,}454{,}464\).

In practice, our algorithm has a time complexity far less than the worst case $O(n2^n)$ bound since graphs typically have less than $2^n$ independent sets. For example, a 5-by-5 grid graph has only $55,447$ independent sets, which is less than $2^{25} = 33,554,432$, meaning the algorithm runs around 1000 times faster than the worst case bound would suggest. In general, counting the number of independent sets in a graph is also difficult and varies from graph to graph. But for typical graphs, the actual time to compute is significantly smaller than $O(n2^n)$. 

Further speed-ups are possible when the graph has nontrivial symmetries. If two independent sets can be mapped to one another by a symmetry of the graph, then they have identical values of \(a(I)\) and \(b(I)\) and contribute equally to the final sum. It therefore suffices to compute these quantities once for each symmetry class of independent sets. For example, in a complete bipartite graph there are only $O(n)$ types of independent sets (up to symmetry), and the values $a(I)$, $b(I)$ can be computed inductively in $O(n)$ total time, giving an overall running time of $O(n)$. 

More generally, there are interesting connections between using inclusion-exclusion and dynamic programming that are worth exploring further \cite{husfeldt2011, karp1982dynamic}. 

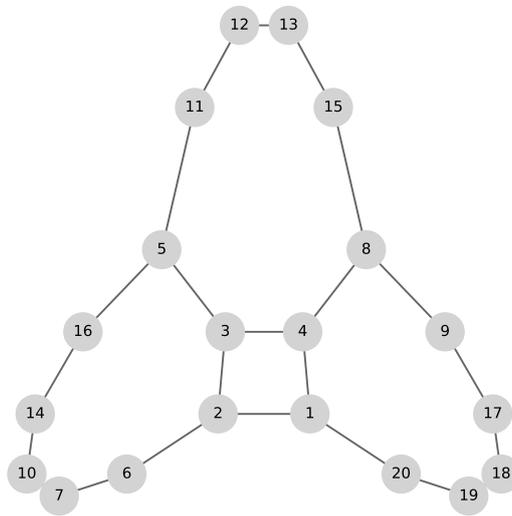
\begin{figure}[t]
    \centering
    \begin{tikzpicture}[scale= 0.5,
        every node/.style={circle, draw=none, fill=gray!40, minimum size=4mm, inner sep=0.5pt, font=\small},
        edge/.style={gray!70, line width=1.15pt, line cap=round, line join=round}]
    
    \node (12) at (4.99,10.43) {12};
    \node (13) at (6.02,10.43) {13};
    \node (11) at (4.04,8.69) {11};
    \node (15) at (6.97,8.69) {15};
    \node (5)  at (3.34,5.66) {5};
    \node (8)  at (7.66,5.66) {8};
    \node (16) at (1.69,3.92) {16};
    \node (3)  at (4.69,3.92) {3};
    \node (4)  at (6.32,3.92) {4};
    \node (9)  at (9.32,3.92) {9};
    \node (14) at (0.68,2.17) {14};
    \node (2)  at (4.53,2.17) {2};
    \node (1)  at (6.47,2.17) {1};
    \node (17) at (10.33,2.17) {17};
    \node (10) at (0.50,0.90) {10};
    \node (6)  at (2.61,0.90) {6};
    \node (20) at (8.39,0.90) {20};
    \node (18) at (10.51,0.90) {18};
    \node (7)  at (1.18,0.43) {7};
    \node (19) at (9.82,0.43) {19};
    
    \draw[edge] (12) -- (13);
    \draw[edge] (12) -- (11);
    \draw[edge] (13) -- (15);
    \draw[edge] (11) -- (5);
    \draw[edge] (15) -- (8);
    \draw[edge] (5) -- (16);
    \draw[edge] (5) -- (3);
    \draw[edge] (16) -- (14);
    \draw[edge] (14) -- (10);
    \draw[edge] (10) -- (7);
    \draw[edge] (7) -- (6);
    \draw[edge] (6) -- (2);
    \draw[edge] (3) -- (2);
    \draw[edge] (3) -- (4);
    \draw[edge] (4) -- (1);
    \draw[edge] (2) -- (1);
    \draw[edge] (4) -- (8);
    \draw[edge] (8) -- (9);
    \draw[edge] (9) -- (17);
    \draw[edge] (17) -- (18);
    \draw[edge] (18) -- (19);
    \draw[edge] (19) -- (20);
    \draw[edge] (20) -- (1);
    \end{tikzpicture}
\caption{A connected graph on $20$ vertices to illustrate the computational advantage of the proposed method over brute-force enumeration of all $20!$ linear orderings and filtering successive orderings.}
\label{fig:graph}
\end{figure}

\section{Successive vertex orderings from a prescribed independent set}
\label{appendix:seeded_svo}

The successive vertex ordering formalism developed in Sections~\ref{sec:introduction} and~\ref{sec:proof} describes a connected growth process beginning from a single initial vertex.  Theorem~\ref{thm:derivative} enumerates orderings according to the number of non-first vertices that appear before all of their neighbours. Such 'bad' vertices may occur at arbitrary positions in the ordering and correspond to the introduction of new disconnected seeds during the growth process. Here we consider a related but distinct question: how many connected growth sequences are possible when the process begins from a prescribed set of $k$ disconnected seed vertices and no additional seeds are introduced thereafter. Such orderings model growth processes that start simultaneously from several disconnected seeds which subsequently expand and merge into a connected structure.

Let \(G=(V,E)\) be a finite connected graph with \(|V|=n\), and let \(S\subseteq V\) be a fixed independent set with \(|S|=k\). 

\begin{definition}
\label{def:seeded}
An \emph{\(S\)-seeded successive vertex ordering} of \(G\) is a linear ordering
\(\pi:V\to\{1,2,\dots,n\}\) such that:
\begin{enumerate}
\item the first \(k\) positions are occupied exactly by the vertices of \(S\), in arbitrary order;

\item for every vertex \(v\in V\setminus S\), there exists a neighbour \(u\sim v\) satisfying \(\pi(u)<\pi(v)\).
\end{enumerate}
\end{definition}

Thus every vertex outside \(S\) must attach to a previously placed vertex. Such a vertex may be adjacent either to a seed vertex in \(S\) or to a vertex that was added after the initial seed set.

We denote by \(\tau(G,S)\) the number of \(S\)-seeded successive vertex orderings of \(G\).

Define \(R:=V\setminus S\), let \(H:=G[R]\) denote the induced subgraph on \(R\), and define
\(
W:=V\setminus N_G[S].
\)
Equivalently, \(W\) consists of those vertices of \(R\) which are not adjacent to any vertex of \(S\). Every vertex in \(R\setminus W\) already has a neighbour in \(S\), and therefore automatically satisfies the successive condition. Hence only vertices in \(W\) can violate the successive condition.

\begin{theorem}
\label{thm:multiseed}
Let \(G=(V,E)\) be a finite connected graph on \(n\) vertices, and let \(S\subseteq V\) be an independent set of size \(k\). Define
\[
R:=V\setminus S,
\qquad
H:=G[R],
\qquad
W:=V\setminus N_G[S].
\]

Then
\[
\tau(G,S)
=
k!(n-k)!
\sum_{\substack{I\subseteq W\\ I\in\mathcal I(H)}}
(-1)^{|I|}
\,b_H(I),
\]
where \(b_H(I)\) is defined recursively by
\[
b_H(\emptyset)=1,
\]
and
\[
b_H(I)
=
\frac{1}{|N_H[I]|}
\sum_{v\in I}
b_H(I\setminus\{v\}),
\qquad
I\neq\emptyset.
\]
\end{theorem}

\begin{proof}
The vertices of \(S\) may appear in any order within the first \(k\) positions, giving \(k!\) possible orderings. Fix one such ordering. We then choose uniformly at random an ordering of the remaining vertices \(R=V\setminus S\). Since \(|R|=n-k\), there are \((n-k)!\) such orderings.

For each vertex \(v\in W\), define the bad event
\[
B_v
=
\{
\pi(v)<\pi(u)
\text{ for all }
u\in N_H(v)
\}.
\]
For a subset \(I\subseteq W\), define
\[
B_I:=\bigcap_{v\in I} B_v.
\]

Every vertex in \(R\setminus W\) already has a neighbour in \(S\), and therefore automatically satisfies the successive condition. Hence an ordering is valid if and only if none of the bad events \(B_v\) occur for \(v\in W\). Therefore, by inclusion--exclusion,
\[
\Pr(\text{valid ordering})
=
\Pr\!\left(\bigcap_{v\in W}\overline{B_v}\right)
=
\sum_{I\subseteq W}
(-1)^{|I|}
\Pr(B_I).
\]

Since \(W\subseteq R\), every set \(I\subseteq W\) is contained in the vertex set of \(H=G[R]\). If \(I\) contains adjacent vertices in \(H\), then \(B_I\) is impossible, so the sum reduces to independent sets of \(H\):
\[
\Pr(\text{valid ordering})
=
\sum_{\substack{I\subseteq W\\ I\in\mathcal I(H)}}
(-1)^{|I|}
\Pr(B_I).
\]

For an independent set \(I\subseteq W\), let
\[
p_H(I):=\Pr(B_I).
\]
The derivation of \(\Pr(B_I)\) is identical to that in the proof of
Theorem~\ref{thm:main}, with the only difference that the vertices of
\(S\) are already fixed in the first \(k\) positions. Consequently there
is no analogue of the factor \(a(I)/n\).

Repeating the same argument yields
\[
p_H(I)
=
\sum_{\rho\in S_I}
\prod_{j=1}^{|I|}
\frac{1}
{|N_H[\{\rho_j,\dots,\rho_{|I|}\}]|},
\]
where \(S_I\) denotes the set of permutations of \(I\).
Partitioning the permutations according to their first element gives
\[
p_H(I)
=
\frac{1}{|N_H[I]|}
\sum_{v\in I}
p_H(I\setminus\{v\}),
\]
with \(p_H(\emptyset)=1\).

Since \(b_H(I)\) satisfies the same recursion and initial condition,
it follows by induction on \(|I|\) that
\[
p_H(I)=b_H(I)
\]
for every independent set \(I \subseteq W\).

\end{proof}

\begin{proposition}\label{prop:seed-sum}
For any finite connected graph $G=(V,E)$,
\[
\sigma(G)=\sum_{s\in V}\tau(G,\{s\}).
\]
\end{proposition}

\begin{proof}
Every successive vertex ordering $\pi$ of $G$ has a unique first vertex
$s=\pi^{-1}(1)$. For this $s$, the ordering $\pi$ satisfies
Definition~\ref{def:seeded}: the seed $s$ occupies position~$1$, and every
subsequent vertex has a neighbour appearing earlier. Hence $\pi$ is counted
by $\tau(G,\{s\})$.

Conversely, every ordering counted by $\tau(G,\{s\})$ is a successive ordering
of $G$ with first vertex $s$. The sets of orderings counted by
$\tau(G,\{s\})$ for distinct $s\in V$ are therefore pairwise disjoint, and
their union is the set of all successive orderings of $G$. This gives the
claimed identity.
\end{proof}

\section{Proof of vertex deletion Theorem~\ref{thm:vertex_set_deletion}}
\label{appendix:vertex_deletion}

Here we present the explicit verification of the polynomial decomposition and neighbourhood identities under vertex set deletion.

\begin{proof}
We begin with the polynomial decomposition.
By definition,
\[
P_G(x)
=
\sum_{I\in\mathcal I(G)} w_G(I)x^{|I|}
\]
Partition the independent sets of $G$ according to whether they
intersect $S$ or avoid $S$. This gives
\[
P_G(x)
=
\sum_{\substack{I\in\mathcal I(G)\\ I\cap S=\varnothing}}
w_G(I)x^{|I|}
+
\sum_{\substack{I\in\mathcal I(G)\\ I\cap S\neq\varnothing}}
w_G(I)x^{|I|}
\]
The second sum is precisely $U_S(x)$.
For the first sum, every independent set avoiding $S$
lies in $\mathcal I(G')$, so
\[
\sum_{\substack{I\in\mathcal I(G)\\ I\cap S=\varnothing}}
w_G(I)x^{|I|}
=
\sum_{I\in\mathcal I(G')} w_G(I)x^{|I|}
\]
Adding and subtracting the weights $w_{G'}(I)$ yields
\[
\sum_{I\in\mathcal I(G')} w_G(I)x^{|I|}
=
\sum_{I\in\mathcal I(G')} w_{G'}(I)x^{|I|}
-
\sum_{I\in\mathcal I(G')}
\bigl(w_{G'}(I)-w_G(I)\bigr)x^{|I|}
\]
The first term equals $P_{G'}(x)$ and the second is $R_S(x)$,
which proves
\[
P_G(x)=P_{G'}(x)-R_S(x)+U_S(x)
\]

\medskip

(i)
Independence of a set depends only on edges whose endpoints lie inside
the set. Since $G'$ is obtained from $G$ by deleting the vertices in $S$
together with all incident edges, the adjacency relations among
vertices in $V\setminus S$ remain unchanged. Thus a subset
$I\subseteq V\setminus S$ is independent in $G'$ if and only if it is
independent in $G$, which establishes the stated description of
$\mathcal I(G')$.

\medskip

(ii)
For $I\subseteq V\setminus S$ we have
\[
N_{G'}[I]=N_G[I]\cap (V\setminus S)
\]
Taking cardinalities gives
\[
|N_{G'}[I]|
=
|N_G[I]|-|N_G[I]\cap S|
\]
Since $|V(G')|=n-|S|$, it follows that
\[
a_{G'}(I)
=
(n-|S|)-|N_{G'}[I]|
=
n-|N_G[I]|-\bigl(|S|-|N_G[I]\cap S|\bigr),
\]
which yields the claimed relation.

\medskip

(iii)
The defining recurrences for $b_G$ and $b_{G'}$ are
\[
|N_G[I]|\,b_G(I)
=
\sum_{v\in I} b_G(I\setminus\{v\}),
\qquad
|N_{G'}[I]|\,b_{G'}(I)
=
\sum_{v\in I} b_{G'}(I\setminus\{v\})
\]
Taking the difference between these identities and writing
$b_{G'}(I)=b_G(I)+\Delta b_I$ gives
\[
|N_{G'}[I]|\Delta b_I
=
\sum_{v\in I}\Delta b_{I\setminus\{v\}}
+
\bigl(|N_G[I]|-|N_{G'}[I]|\bigr)b_G(I).
\]
By part (ii), the difference $\bigl(|N_G[I]|-|N_{G'}[I]|\bigr)$ equals $|N_G[I]\cap S|$,
which yields \eqref{eq:multi_delta_b}.

For nonempty $I$, we have $|N_{G'}[I]|\ge |I|>0$, so we may divide by $|N_{G'}[I]|$ to express $\Delta b_I$ in terms of $b_G(I)$ and the values $\Delta b_J$ with $J\subsetneq I$. This determines $\Delta b_I$ uniquely by induction on $|I|$, similar to solving a triangular system of equations.
\end{proof}

\section{A leaf recursion and universal lower bound for $\sigma(G)$}

For general connected graphs, the number of successive vertex orderings is given by the inclusion--exclusion formula of Theorem~\ref{thm:main}. For trees, however, the enumeration of successive vertex orderings admits a simple recursion on the leaf structure, avoiding the inclusion--exclusion machinery altogether.

\begin{lemma}[Leaf recursion]
\label{lem:leafrecursion}
Let T be a tree on $n\ge 2$ vertices. Then
\[
\sigma(T)=\sum_{\ell\in L(T)} \sigma(T-\ell),
\]
where \(L(T)\) denotes the set of leaves of \(T\).
\end{lemma}

\begin{proof}
In any successive ordering of a tree, the last vertex must be a leaf. Indeed, let \(v\) be the final vertex of a successive ordering of \(T\). Then the first \(n-1\) vertices induce the graph \(T-v\). Since every initial segment of a successive ordering induces a connected subgraph, \(T-v\) must be connected. Removing a non-leaf vertex from a tree disconnects it, so \(v\) must be a leaf.

The successive orderings of \(T\) are therefore partitioned according to their final vertex. Fix a leaf \(\ell\in L(T)\). If a successive ordering of \(T\) ends in \(\ell\), then deleting \(\ell\) yields a successive ordering of \(T-\ell\). Conversely, every successive ordering of \(T-\ell\) extends to a successive ordering of \(T\) by placing \(\ell\) last. This is valid because the unique neighbour of \(\ell\) belongs to \(T-\ell\) and therefore appears earlier than \(\ell\) in the resulting ordering. These two operations are inverse to one another, and hence define a bijection between successive orderings of \(T-\ell\) and successive orderings of \(T\) ending in \(\ell\).

Summing over all leaves \(\ell\) gives the result.
\end{proof}

Lemma~\ref{lem:leafrecursion} expresses \(\sigma(T)\) recursively in terms of smaller trees obtained by deleting leaves. Iterating the recursion shows that \(\sigma(T)\) may be viewed as a sum over all valid sequences of leaf deletions of \(T\). As a first application, we obtain a universal lower bound for \(\sigma(G)\) valid for every connected graph.

\begin{proposition}
\label{prop:lowerbound}
Let \(G\) be a connected graph on \(n\) vertices. Then
\[
\sigma(G)\ge 2^{n-1}.
\]
\end{proposition}

\begin{proof}
Let \(T\) be a spanning tree of \(G\). Every successive ordering of \(T\) is also a successive ordering of \(G\), since any connected prefix in \(T\) remains connected in the supergraph \(G\). Hence
\[
\sigma(G)\ge \sigma(T),
\]
so it suffices to prove the bound for trees.

We proceed by induction on \(n\). The case \(n=1\) is immediate. Let \(T\) be a tree on \(n\ge2\) vertices. By Lemma~\ref{lem:leafrecursion},
\[
\sigma(T)
=
\sum_{\ell\in L(T)}
\sigma(T-\ell).
\]
For every leaf \(\ell\), the graph \(T-\ell\) is a tree on \(n-1\) vertices. By the induction hypothesis,
\[
\sigma(T-\ell)\ge 2^{n-2}.
\]
Since every tree has at least two leaves,
\[
\sigma(T)
=
\sum_{\ell\in L(T)}
\sigma(T-\ell)
\ge
2\cdot 2^{n-2}
=
2^{n-1}.
\]
This proves the claim.
\end{proof}

\begin{remark}
The lower bound is sharp. Indeed, if equality holds in Proposition~\ref{prop:lowerbound} for a tree \(T\), then
\[
\sigma(T)
=
\sum_{\ell\in L(T)}
\sigma(T-\ell)
=
2^{n-1}.
\]
Since each term satisfies \(\sigma(T-\ell)\ge 2^{n-2}\), if \(|L(T)|\ge 3\) then
\[
\sigma(T)
=
\sum_{\ell\in L(T)}
\sigma(T-\ell)
\ge
3\cdot 2^{n-2}
>
2^{n-1},
\]
which is impossible. Hence equality forces \(|L(T)|=2\). A tree with exactly two leaves is a path, so \(T\cong P_n\).

For the converse, let \(T=P_n\). The vertices appearing in any successive ordering always form a contiguous subpath. This follows by induction: the first vertex forms a trivial subpath, and each subsequent vertex must be adjacent to the current subpath, so it can only be attached at one of its two endpoints. Equivalently, working backwards from the full path, one may repeatedly delete either the left or right endpoint. At each of the \(n-1\) deletion steps there are exactly two choices, giving
\[
\sigma(P_n)=2^{n-1}.
\]
\end{remark}

\bibliographystyle{alphaurl}
\bibliography{references}

\end{document}